%% file: main.tex
\documentclass[]{article}

\input{preamble}

\title{Globally stable cylinders for hyperbolic CAT(0) cube complexes}
\author{Nir Lazarovich\thanks{Supported by the Israel Science Foundation (grant no. 1562/19), and by the German-Israeli Foundation for Scientific Research and Development.} ~and Michah Sageev\thanks{Supported by the Israel Science Foundation (grant no. 660/20)}}

\begin{document}

\maketitle
\begin{abstract}
    Rips and Sela \cite{rips1995canonical} introduced the notion of globally stable cylinders and asked if all Gromov hyperbolic groups admit such. We prove that hyperbolic cubulated groups admit globally stable cylinders. 
\end{abstract}

\paragraph{Globally stable cylinders.}

Rips and Sela \cite{rips1995canonical} introduced the notion of globally stable cylinders in their work on solutions of equations over groups. In the context of a $\delta$-hyperbolic group, the idea is as follows. Given two points  $x$ an $y$ in a $\delta$-hyperbolic space $X$, one can choose a geodesic $[x,y]$ joining them. Then given three vertices $x,y$ and $z$, one has that the two geodesics $[x,y]$ and $[x,z]$ ``fellow travel'' (i.e. are within $\delta$ of each other) up to some median point from which they diverge. The idea of stable cylinders is to thicken the geodesics into ``cylinders'' which are not just close to one another, but actually agree for most of the time they fellow travel. 

More precisely,  let $\theta\ge 0$. A $\theta$-cylinder $C(x,y)$ of $x,y\in X$ is a subset that satisfies $[x,y]\subseteq C(x,y) \subseteq N_\theta ([x,y])$ for every geodesic $[x,y]$ connecting $x,y$.

A choice of $\theta$-cylinders $C:X\times X\to 2^X$ for every $x,y\in X$ is called globally stable if there exist $k,R\ge 0$ such that:
\begin{enumerate}
    \item \emph{inversion invariance:} $C(x,y)=C(y,x)$ for all $x,y\in X$, and
    \item \emph{$(k,R)$-stability:} for all $x,y,z\in X$ there exists $k$ $R-$balls $B_1,\ldots,B_k$ in $X$ such that 
    \begin{equation}\label{eq: stability}\tag{$\ast$}
        (C(x,y)\cap B(x,\rho)) - 
    \bigcup_{i=1}^k B_i =( C(x,z)\cap B(x,\rho) )- \bigcup_{i=1}^k B_i
    \end{equation} 
    where $\rho=(y.z)_x= \frac{1}{2}(d(x,y)+d(x,z) - d(y,z))$ is the Gromov product (see Figure \ref{fig:stable_rep}).
\end{enumerate}
\begin{figure}
    \centering
    \includegraphics[scale=0.6]{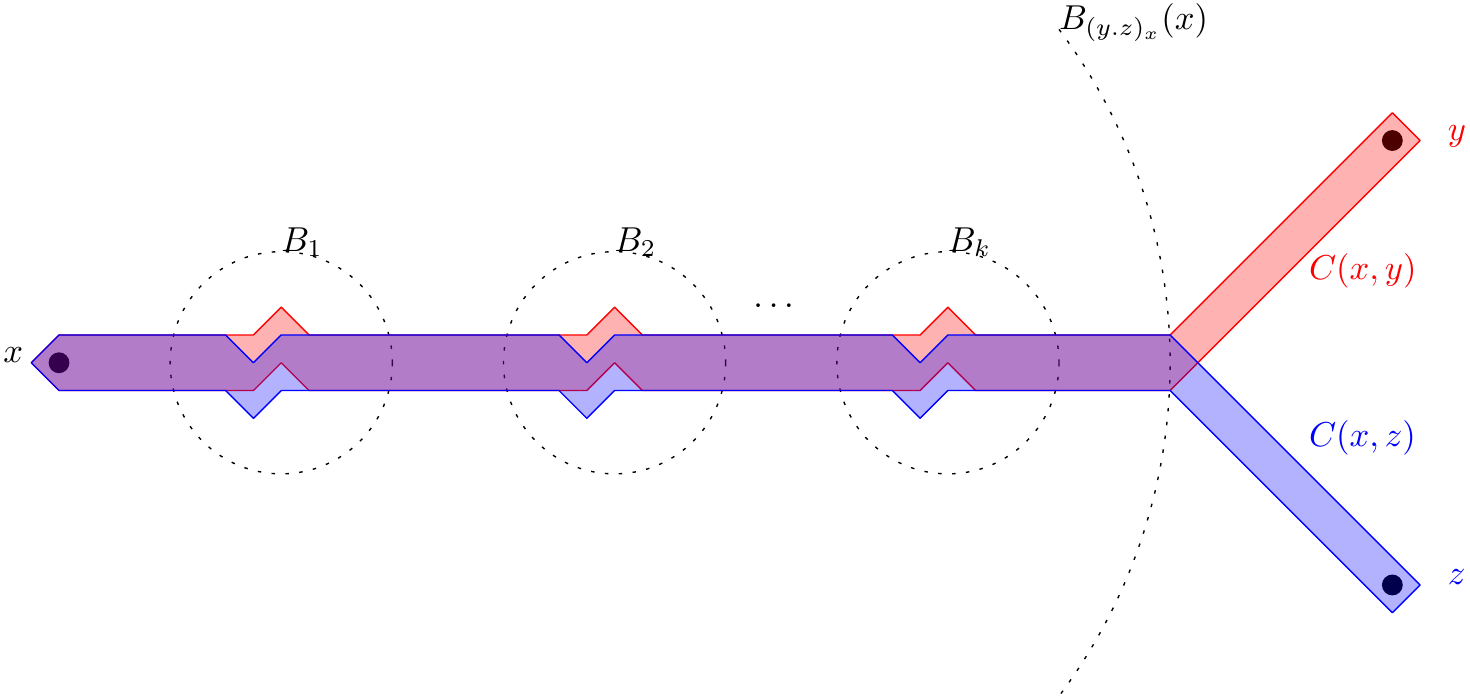}
    \caption{$(k,R)$-uniformly stable cylinders}
    \label{fig:stable_rep}
\end{figure}

Let $G$ be a hyperbolic group. The group $G$  \emph{admits globally stable cylinders} if some geodesic hyperbolic space $X$ on which $G$ acts properly cocompactly admits globally stable cylinders which are \emph{$G$-invariant}, i.e, $C(gx,gy) = gC(x,y)$ for all $x,y\in X$ and $g\in G$. For completeness we show in Proposition \ref{prop: stable cylinders under qi} that this is a group property, and does not depend on the space $X$. 

Rips and Sela \cite{rips1995canonical} showed that hyperbolic $C'(1/8)$-small cancellation groups have globally stable cylinders. In fact, they show that such groups are $(1,R)-$stable.
They also asked if all hyperbolic groups admit globally stable cylinders. 
Recently, Kharlampovich-Sklinos \cite{kharlampovich2021first} used the globally stable cylinders of $C'(1/8)$-small cancellation groups to study the first-order theory of random groups.
Globally stable cylinders are also used by the first author \cite{lazarovich2021complexity} to study the connection between complexity and volume of hyperbolic groups.

In this paper, we prove the existence of globally stable cylinders for hyperbolic cubulated groups.

\begin{theorem*}\label{thm: main thm}
Let $X$ be a hyperbolic $d$-dimensional CAT(0) cube complex. Then, $X$ admits globally stable cylinders which are $\Aut(X)$-invariant.

In particular, every hyperbolic cubulated group admits globally stable cylinders.
\end{theorem*}

\begin{remark*}
Note that in the Theorem, $X$ is not required to be locally finite. Moreover the parameters $k$ and $R$ in \eqref{eq: stability} depend only on the hyperbolicity constant $\delta$ and the dimension $d$ of $X$.
\end{remark*}


\paragraph{Preliminaries on CAT(0) cube complexes.}
We briefly recall standard material in CAT(0) cube complexes, as well as introducing some new definition which we will employ. For a more extensive discussion on CAT(0) cube complexes see e.g \cite{bestvina2014geometric}. 
Let $X$ be a hyperbolic $d$-dimensional CAT(0) cube complex. 
Let $\Hyp(X)$ be its set of hyperplanes.
Every hyperplane $\hat{h}$ separates $X$ into two components. In each component there is a unique maximal subcomplex $h$ of $X$ which we call a halfspace bounded by $\hat{h}$.
The halfspaces of $X$ are convex and the collection of all of them is denoted by $\HS(X)$.
The set $\HS(X)$ has a natural involution $h \mapsto h^*$ mapping a halfspace to its complementary halfspace. Moreover, we denote by $\hat{\;}:\HS(X)\to \Hyp(X)$ the map $h\mapsto \hat{h}$ that assigns to a halfspace its bounding hyperplane.
For $x,y\in X^{(0)}$, define the set of halfspaces separating $x$ from $y$ as
$$ \vec\HS(x,y) = \{ k\in \HS(X) | x\notin k \ni y\}, $$
and its corresponding set of hyperplanes by
$$ \Hyp(x,y) = \{ \hat{k}| k\in \HS(x,y)\}.$$
All other hyperplanes are called \emph{peripheral} to $x,y$, and the set of \emph{peripheral halfspaces} is defined as
$$\PHyp(x,y) = \{ h\in \calH(X) | x,y\in h^*\}.$$ 
The interval between $x$ and $y$ is the union of all $\ell_1$ geodesics between $x$ and $y$. It turns out that the interval is also the intersection of all the halfspaces that contain both $x$ and $y$. That is, it is the set $$ I(x,y) = \{z\in X | d(x,y) = d(x,z)+d(z,y)\} = \bigcap_{h\in\PHyp(x,y)} h^*$$ where $d$ denotes the $\ell_1$-metric on $X$.
The interval $I(x,y)$ is a convex subcomplex of $X$, and is isomorphic to the dual cube complex to $\HS(I(x,y))$, the pocset of halfspaces  associated to the set of hyperplanes in $\vec\HS(x,y)$. 
The nearest-point projection $\pi_{I(x,y)}:X\to I(x,y)$ can be expressed using the language of ultrafilters as $\pi_{I(x,y)}(v) = v\cap \HS(I(x,y))$.

\begin{observation} 
\label{obs: geodesic neighborhood}
Note that in the case that $X$ is $\delta$-hyperbolic, any two $\ell_1$-geodesics between $x$ and $y$ are contained in  $\delta$-neighborhoods of one another. Thus the interval $I(x,y)$ is contained in the $\delta$-neighborhood of any $\ell_1$-geodesic joining $x$ and $y$. 
\end{observation}

The median $m(x,y,z)$ of $x,y,z$ in $X$ is the unique point in $I(x,y)\cap I(y,z) \cap I(x,z)$. Alternatively, $m(x,y,z)$ is the vertex $\pi_{I(x,y)}(z)$.

We denote $\hat{h}\pitchfork \hat{k}$ if the hyperplanes $\hat{h},\hat{k}\in\Hyp(X)$ intersect. 
A nested sequence of halfspaces $h_1<\ldots<h_n$ is called a \emph{pencil of halfspaces} and the corresponding sequence of bounding hyperplanes will be called a \emph{pencil of hyperplanes}. 

We now want to refine the notion of peripheral hyperplanes and halfspaces to those that do not run too long parallel to a given interval. For a subset $\calK \subset \HS$ and a halfspace $h\in \HS$ define their intersection number as the maximal size of a pencil of hyperplanes in $\calK$ that intersect $\hat{h}$, i.e, $$i(h,\calK)=\max\{n | \exists k_1 < \ldots < k_n \in \calK :  \hat{h}\pitchfork\hat{k_i} \forall 1\le i\le n\}.$$
Note that $i(h,\calK)$ may be infinite.
We then define the following subset of $\PHyp(x,y)$,
$$\PHyp_D(x,y) = \{ h\in \PHyp(x,y) | i(h,\vec\HS(x,y))\le D\},$$
We call these the \emph{$D$-peripheral halfspaces} of $[x,y]$
and correspondingly 
$$I_D(x,y) = \bigcap_{h\in \PHyp_D(x,y)} h^*.$$

\begin{observation}\label{obs: bounded projection}
If $h\in \PHyp_D(x,y)$ then $\diam(\pi_{I(x,y)} h)\le Dd$. Otherwise, there are points $u,v\in \pi_{I(x,y)}(h)$ of distance $>Dd$. By the pigeonhole principle out of the $Dd+1$ hyperplanes separating $u,v$ there are $D+1$ hyperplanes satisfying $k_1<\ldots<k_{D+1}\in \vec \HS(x,y)$. But, by the properties of the nearest point projection, $\hat k_1,\ldots,\hat k_{D+1}$ are transverse to $\hat h$, contradicting the assumption that $h\in \PHyp_D(x,y)$. 
\end{observation}

Henceforth, we will assume that $X$ is a $\delta$-hyperbolic CAT(0) cube complex of dimension $d$.
A \emph{grid of size $m$} is a pair of pencils of hyperplanes $\hat{h}_1,\ldots, \hat{h}_m$ and $\hat{k}_1,\ldots,\hat{k}_m$ such that $\hat h_i\pitchfork \hat k_j$ for all $1\le i,j \le m$. 

\begin{lemma}\label{lem: maximal grid number}
There exists a maximal number $D=D(\delta)$ such that $X$ contains a grid of size $D$. 
\end{lemma}

\begin{proof}
Assume $h_1 <\ldots< h_D\in \HS$ and $k_1<\ldots<k_D\in \HS$ satisfy $\hat h_i\pitchfork \hat k_j$ for all $1\le i,j \le D$.
Let $v_1,v_2,v_3,v_4$ be vertices of $X$ in $h_D \cap k_D, h_1^* \cap k_D, h_1^* \cap k_1^*, h_D \cap k_1^*$.
Consider a geodesic quadrilateral $v_1v_2v_3v_4$. Both pairs of opposite sides are separated by at least $D$ hyperplanes. Since $X$ is $\delta$-hyperbolic, every quadrilateral has to be $2\delta$-slim. Thus, $D\le 2\delta$. 
\end{proof}

Henceforth, let $D$ be as in the Lemma \ref{lem: maximal grid number}. 
Our goal is to prove that the choice $C(x,y) = I_D(x,y)$ for all $x,y\in X$ is a $(k,R)$-uniformly stable choice of $\theta$-cylinders for some fixed $k,R,\theta$ that depend only on $X$ (in fact, only on $d$ and $D$).

We begin by showing that $I_D(x,y)$ are $\theta$-cylinders. First note, that by Observation \ref{obs: geodesic neighborhood}, $I(x,y)\subset N_\delta([x,y])$ for any geodesic $[x,y]$. Thus, to prove that $I_D(x,y)$ are $\theta$-cylinders, it suffices to prove the following lemma. 

\begin{lemma}\label{lem: cylinders}
Let $\theta = Dd$. Then $I(x,y) \subset I_D (x,y) \subseteq N_\theta(I(x,y))$.
\end{lemma}

\begin{proof}
By definition $I(x,y) \subset I_D(x,y)$.
To prove the second inclusion, assume for contradiction that there exists a vertex $w\in I_D(x,y)$ for which the distance $d(w,I(x,y))>\theta$. Since $I(x,y)$ is a convex subcomplex, this means that $w$ is separated by more than $\theta=Dd$ hyperplanes from $I(x,y)$. 
By the pigeonhole principle there are $D+1$ halfspaces $h_1<\ldots<h_{D+1}$ such that $I(x,y) \subseteq h_1$ and $w \notin h_{D+1}$. By definition of $I_D(x,y)$ the hyperplane $h_{D+1}$ is not in $\PHyp_{D}(x,y)$  and thus there exist $k_1<\ldots<k_{D+1}\in \vec \HS(x,y)$ such that $\hat h_{D+1} \pitchfork \hat k_j$ for all $1\le j\le D$. Since each $\hat k_i$ intersects $I(x,y)$, and $\hat h_i$ separates $I(x,y)$ from $\hat h_{D+1}$ it follows that $\hat h_i \pitchfork \hat k_j$ for all $1\le i,j \le D+1$. A contradiction to the choice of $D$.
\end{proof}

Let 
$$\PHyp_{D,\rho}(x,y) = \{ h\in \PHyp_D(x,y) | \dmax (x,\pi_{I(x,y)}(h)) < \rho\}$$
where $\dmax(x,A) = \max\{d(x,a)|a\in A\}$. Roughly speaking, this is the collection of $D$-peripheral halfspaces whose projection to $[x,y]$ is $\rho$-close to $x$. 

The key step in proving the existence of globally stable cylinders is to bound the number of projections of $D$-peripheral halfspaces for $[x,y]$ which are not $D$ peripheral for $[x,z]$. More precisely, we have the following.

\begin{proposition}\label{prop: bounded hyperplane difference}
There exists  $M = M(\delta,d)\ge 0$ such that for all $x,y,z\in X$ 
$$|\{ \pi_{I(x,y)} (h) \mid h\in \PHyp_{D,\rho-Dd}(x,y) - \PHyp_{D}(x,z)\}|\le M$$  
where $\rho = (y.z)_x$.
\end{proposition}

\begin{proof}
First, we show that $\PHyp_{D,\rho-Dd}(x,y)$ is contained in the full peripheral set of hyperplanes for the pair $(x,z)$. 

\begin{claim}\label{claim: peripheral containment}
$\PHyp_{D,\rho-Dd}(x,y)\subset \PHyp(x,z)$.
\end{claim}

\begin{proof}
Note that $\rho = d(x,m)$ where $m=m(x,y,z)$ is the median of $x,y,z$ in $X$.
Assume for contradiction that $h\in  \PHyp_{D,\rho-Dd}(x,y) - \PHyp(x,z)$. 
It follows that $x\in h^*$ but the hyperplane $\hat{h}$ separates $x,z$, so $z\in h$.
Whence $$m= \pi_{I(x,y)}(z) \in \pi_{I(x,y)}(h).$$
But then $$\rho = d(x,m) <\dmax (x,\pi_{I(x,y)}(h))<\rho-Dd$$ a contradiction.
\end{proof}

To prove the proposition, we will assume that the number of projections $\pi_{I(x,y)} (h)$ of halfspaces in  $h\in \PHyp_{D,\rho-Dd}(x,y) - \PHyp_{D}(x,z)$ is large. Our strategy will be to construct a grid of hyperplanes of size greater than $D$ which contradicts the hyperbolicity of $X$. 
In carrying out this strategy, we will need to make use of the following two technical ``pigeonhole claims", which ensure that our objects of interest are disjoint. We abuse notation and let $\bbR^d$ denote the standard cubulation of $\bbR^d$. 

\begin{claim}\label{claim: distinct to disjoint subcomplexes}
For every $m,R,d$ there exists $N_1=N_1(m,R,d)$ such that for every $N_1$ distinct subcomplexes of $\bbR^d$ of diameter $\le R$  there is subcollection of $m$ subcomplexes  $A_1,\ldots,A_{m}$ and a pencil of $m-1$ halfspaces $s_1<\ldots< s_{m-1}$ such that $A_i \subset  s_{i-1}^* \cap s_i  $ for all $1\le i <m$.
\end{claim}

\begin{proof}
Let $K=K(d,R)$ be the number of subcomplexes in $[0,R]^d$.
Set $N_1 = m(R+1)^2dK$.
Let $\calA$ be a collection of $N_1$ distinct subcomplexes of $\bbR^d$ of diameter $\le R$. Projecting the subcomplexes to each of the standard axes , we get intervals of length at most $R$. Thus, at most $K$ subcomplexes in $\calA$ can have the same projections on all axes. By the Pigeonhole Principle, there exists a standard axis $\ell$ such that $\frac{N_1}{dK} = m(R+1)^2$ of them have distinct projections to $\ell$. Since each projection is an interval of length $\le R$, at least $\frac{m(R+1)^2}{(R+1)^2} = m$ of $\calA$ have disjoint projections to $\ell$. These projections are separated by $m-1$ midpoints of edges in $\ell$, and the subcomplexes are thus separated by the pencil of $m-1$ hyperplanes perpendicular to $\ell$ at those midpoints, as required. 
\end{proof}

\begin{claim}\label{claim: pigeonhole for posets}
Let $(P,\le)$ be a partially ordered set such that at most $d$ elements are pairwise incomparable. For all $m$ there exists $N_2=N_2(m,d)$ such that every sequence of length $N_2$ of elements in $P$ has a subsequence of length $m$ which is strictly monotonic or constant.
\end{claim}

\begin{proof}
Set $N_2 = dm^3$.
By Dilworth's Theorem \cite{dilworth2009decomposition}, $P$ can be partitioned to $d$ linearly ordered subsets. 
By the Pigeonhole Principle there is a subsequence of length $m^3$ whose elements are totally ordered. 
If at least $m$ of them are equal, we are done. Thus, we can pass to a further subsequence of length $m^2$ of distinct elements.
A standard application of the Pigeonhole Principle shows that one can now find a strictly monotonic subsequence.
\end{proof}

Returning to the proof of Proposition \ref{prop: bounded hyperplane difference}, recall that we will need to apply the pigeonhole claims to a collection of $M$ halfspaces to obtain a disjointness property for large enough $M$. Specifically, let $L=2D+2$, let 
$$T = N_2(N_2(\ldots(N_2(L,d)\ldots,d),d)$$  be the $D+1$ fold application of $N_2$, and finally let $M = N_1(T,Dd,d).$

Let $h^1,\ldots,h^M$ be halfspaces in $\PHyp_{D,\rho-Dd}(x,y) - \PHyp_{D}(x,z)$ with distinct projections $\pi_{I(x,y)}(h^i)$. 
By \cite{brodzki2009propertya}, the interval $I(x,y)$ is isomorphic to a subcomplex of $\bbR^d$. 
By Observation \ref{obs: bounded projection}, the projections $\pi_{I(x,y)}(h^i)$ to $I(x,y)$ have diameter $\le Dd$.  
Thus, by Claim \ref{claim: distinct to disjoint subcomplexes} we can find $T$ halfspaces among $h^1,\ldots,h^M$, without loss of generality $h^1,\ldots,h^{T}$, which have pairwise disjoint projections to $I(x,y)$ and which are separated by a pencil of hyperplanes. That is, there exist $s^1 < \ldots < s^{T-1}$  in $\vec\HS(x,y) \cap \vec\HS(x,z)$ such that $\pi_{I(x,y)}(h^i) \subset s^i \cap (s^{i-1})^*$. This is equivalent to  $(h^i)^* < s^i < h^{i+1}$ for all $1\le i < T$.

For each $1\le i \le T$, the halfspace $h^i$  is in $\PHyp(x,z)$ by Claim \ref{claim: peripheral containment}, but  not in $\PHyp_D(x,z)$ by assumption. Hence there exist $k^i_1<\ldots<k^i_{D+1}\in \vec\HS(x,z)$ such that $\hat h^i\pitchfork \hat k^i_j$ for all $1\le j \le D+1$. But since $h^i \in \PHyp_D(x,y)$, at least one of $k^i_1<\ldots<k^i_{D+1}$ is not in $\vec\HS(x,y)$. Namely $k^i_1 \in\HS(y,z)$ for all $1\le i\le T$.

\begin{figure}
    \centering
    \includegraphics[scale=0.7]{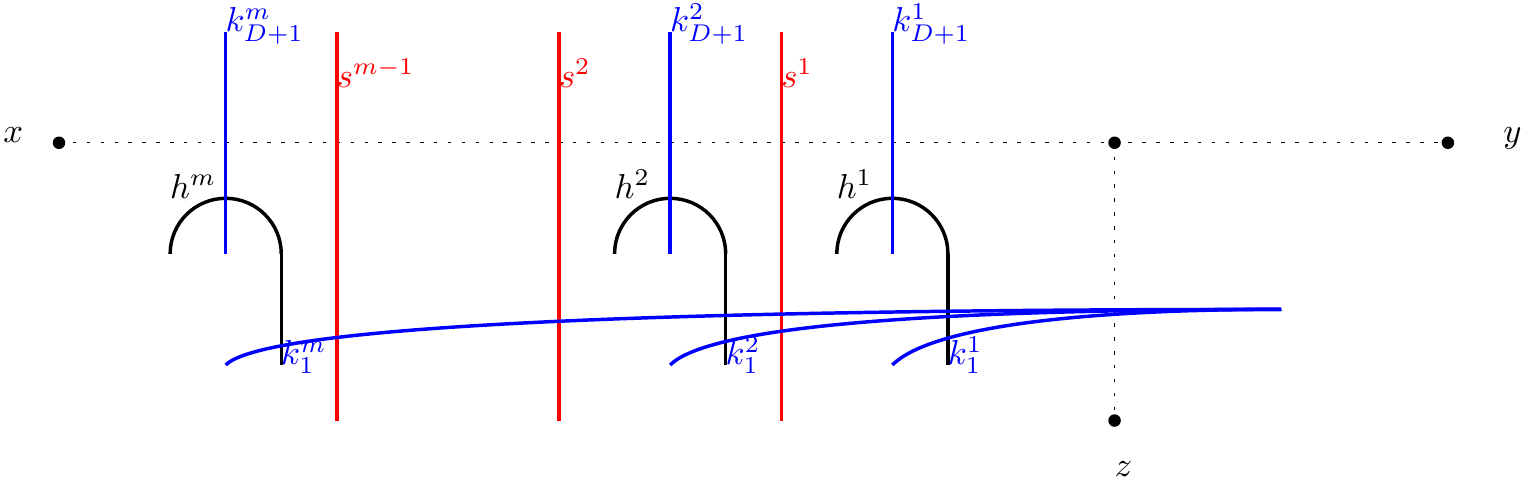}
    \caption{The hyperplane arrangement for the proof of Proposition \ref{prop: bounded hyperplane difference}.}
    \label{fig:my_label}
\end{figure}

By $D+1$ applications of Claim \ref{claim: pigeonhole for posets} we can pass to a subsequence of length $L$, which by abuse of notation we will simply denote $h^1,\ldots,h^L$, for which the sequences of halfspaces $k^i_j$ are constant or strictly monotonic in $i$. That is, for each $1\le j\le D+1$ either 
$$
k^1_j < \ldots < k^L_j, \quad k^1_j > \ldots > k^L_j, \quad \text{ or }  \quad k^1_j = \ldots = k^L_j.
$$

Let us show by induction on $j$ that (up to changing the halfspaces $k^i_j$ if necessary) we may assume that $k^1_j=\ldots=k^L_j$ for all $1\le j\le D+1$ in addition to the properties: $k^i_1 < \ldots < k^i_{D+1}$  and $\hat h^i \pitchfork \hat k^i_j$. 

For the base of the induction, let us show that $k^1_1=\ldots=k^L_1$. Assume for contradiction that $k^1_1,\ldots, k^L_1$ are distinct. Since each $k^i_1$ intersects $h^i$ on the one hand, and separates $y$ and $z$ on the other, it must intersect $s^{i'-1}$ for all $i'<i$. Since $L>2D+1$ this gives a grid of hyperplanes of size $D+1$ which contradicts the definition of $D$.
Thus, $k^1_1=\ldots=k^L_1$.

For the induction step, let $1\le j \le D$ and assume that $k^1_j=\ldots=k^L_j$. Let us prove that we may choose $k^i_{j+1}$ such that $k^1_{j+1}=\ldots=k^L_{j+1}$. Assume that $k^1_{j+1} < \ldots < k^L_{j+1}$. Then $k^{i}_{j}= k^1_j <k^1_{j+1} < k^{i}_{j+1}$. It follows that $$k^i_1<\ldots k^i_j < k^1_{j+1} < k^i_{j+2} <\ldots k^i_{D+1}$$ and $\hat{h}^i \pitchfork \hat{k}^1_{j+1} $ (since $k^i_{j} < k^1_{j+1} < k^{i}_{j+1}$ and both $\hat k^i_j\pitchfork \hat h^i \pitchfork \hat k^i_{j+1}$). This means that we can replace $k^{i}_{j+1}$ for all $1 \le  i \le L$ by $k^1_{j+1}$, and get $k^1_{j+1}=\ldots=k^L_{j+1}$ as desired. 
Similarly, if $k^1_{j+1} > \ldots > k^L_{j+1}$ we can replace $k^{i}_{j+1}$ for all $1 \le  i \le L$ by $k^L_{j+1}$ and get $k^1_{j+1}=\ldots=k^L_{j+1}$ as desired.

This shows that we may assume that $k^1_j = \ldots = k^L_j=:k_j$ for all $1\le j\le D+1$. However, since $\hat k_1,\ldots,\hat k_{D+1}$ intersect all the hyperplanes $\hat h^1,\ldots, \hat h^t$ they must intersect also the separating hyperplanes $\hat s^1,\ldots,\hat s^{t-1}$. As $L>D+2$ the sets $k_1<\ldots<k_{D+1}$ and $s^1<\ldots<s^{t-1}$ form a grid of size $D+1$ which contradicts the definition of $D$.
\end{proof}



\begin{proof}[Proof of Theorem]
Define $C:X\times X\to 2^X$ by $C(x,y) := I_D(x,y)$. 

By Lemma \ref{lem: cylinders} and Observation \ref{obs: geodesic neighborhood}, $C(x,y)$ are $\theta$-cylinders for $\theta=Dd+\delta$. 
It is clear from the definition that this choice of cylinders is $\Aut(X)$-invariant and inversion invariant. 
It remains to show that it is $(k,R)$-uniformly stable.


Set $R=5Dd$.
Let $$ P=\{m\} \cup \bigcup \pi_{I(x,y)} (h)$$ where $m=m(x,y,z)$, and the union ranges over all $h \in \PHyp_{D,\rho-Dd}(x,y) - \PHyp_{D}(x,z)$. By Observation \ref{obs: bounded projection} and Proposition \ref{prop: bounded hyperplane difference},  $P$ has at most $k=k(d,\delta)$ points. 
Let $\calB = \bigcup_{p\in P} B(p,R)$ be the union of the $R$-balls with centers in $P$.
By symmetry it suffices to show that $C(x,z)  - C(x,y) \cap B(x,\rho) \subseteq \calB$.

Assume $w\in (C(x,z) - C(x,y)) \cap B(x,\rho) - B$ where $B=B(m,R)$ is the ball of radius $R$ around $m$. Let $[x,z]$ be a geodesic passing through $m$. Since $C(x,z)$ is a $\theta$-cylinder, there exists a point $w'\in [x,z]$ such that $d(w,w')\le \theta$. Since $w\in B(x,\rho)$ and $w\notin B(m,R)$ we have that $d(x,w') \le  \rho - R+\theta$, and in particular $w'$ belongs to $C(x,y)$.

Let $w'' = \pi_{I(x,y)}(w)$. Since $w'\in C(x,y)$ we have $d(w,w'')\le d(w,w')\le \theta$, and it follows that $d(x,w'') \le d(x,w')+d(w',w)+d(w,w'') \le \rho - R+3\theta$.

Now, there exists $h\in \PHyp_{D}(x,y) - \PHyp_{D}(x,z)$ such that $w\in h$. By Observation \ref{obs: bounded projection}, $\diam(\pi_{I(x,y)}(h))\le Dd$. Since $w''\in \pi_{I(x,y)}(h)$, $$\dmax(x,\pi_{I(x,y)}(h))\le d(x,w'') + Dd \le  \rho - R+3\theta +Dd \le \rho-Dd.$$  Hence $h\in \PHyp_{D,\rho-R}(x,y) - \PHyp_{D,\rho}(x,z)$ and $w'' \in P$.
Since $d(w,w'')\le \theta = Dd \le R$ we get that $w\in B(w'',R) \subseteq \calB$ as desired.
\end{proof}

\paragraph{Stable cylinders for cubulated groups.} As discussed in the introduction, stability of cylinders does not depend on the space on which $G$ acts properly and cocompactly.

\begin{proposition}\label{prop: stable cylinders under qi}
Let $G$ be a hyperbolic group. Let $X,X'$ be graphs on which $G$ acts properly and cocompactly. The space $X$ admits $G$-globally stable cylinders if and only if $X'$ does.
\end{proposition}

\begin{proof}
Assume $X$ has $G$-globally stable cylinders.

Let $F$ be a finite fundamental domain for $G\actson X$. Define the map $\phi:X \to 2^{X'}$ by sending each $x\in F$ to a non-empty set $\phi(x)$ of diameter $\diam(\phi(x))\le D$ stabilized by the finite group $\Stab_G(x)$, and extend to $X$ $G$-equivariantly.
By the Milnor-Svarc Lemma the coarse map $\phi$ is a quasi-isometry. By replacing $\phi(x)$ by $N_r(\phi(x))$ for some large enough $r$, we assume that $\phi$ is ``surjective'' on $X'$ in the sense that for all $x'\in X'$ there exists $x\in X$ such that $x'\in \phi(x)$. 

By the Morse Lemma for $X'$, let $R$ be the constant such that $N_R(\phi([x,y]))$ contains any geodesic between $x'\in \phi(x)$ and $y'\in \phi(y)$. Define the cylinder  $C'(x',y')$ for $x',y'\in X'$  by $$C'(x',y') = N_R(\bigcup\{ \phi(z)\in X' \mid \forall x,y,z : x'\in \phi(x), y' \in \phi(y), z\in C(x,y)\}).$$
It is straightforward to verify that $C'(x',y')$ are $G$-invariant globally stable cylinders for $X'$.
\end{proof}

\bibliographystyle{plain}
\bibliography{biblio}
\end{document}

%% file: preamble.tex
\usepackage[leqno]{amsmath}
\usepackage[a4paper, total={5.1in, 8in}]{geometry}
\usepackage{amssymb,amsfonts,xfrac,MnSymbol}
\usepackage{amsthm}
\usepackage{cite}
\usepackage{hyperref}
\usepackage{todonotes}
\usepackage{enumitem}
\usepackage{graphicx}
\usepackage{tikz-cd, tikz}
\usepackage{color}
\usepackage{import}

\numberwithin{equation}{section}

\newtheorem{theorem}{Theorem} 
\newtheorem{claim}[theorem]{Claim}
\newtheorem{proposition}[theorem]{Proposition}
\newtheorem{lemma}[theorem]{Lemma}

\newtheorem*{theorem*}{Theorem}
\newtheorem*{claim*}{Claim}
\newtheorem*{proposition*}{Proposition}
\newtheorem*{lemma*}{Lemma}
\newtheorem*{corollary*}{Corollary}

\theoremstyle{definition}

\newtheorem{observation}[theorem]{Observation}

\newtheorem*{definition*}{Definition}
\newtheorem*{observation*}{Observation}
\newtheorem*{remark*}{Remark}
\newtheorem*{example*}{Example}
\newtheorem*{question*}{Question}
\newtheorem*{exercise*}{Exercise}
\newtheorem*{fact*}{Fact}
\newtheorem*{notation*}{Notation}

\newcommand{\bbR}{\mathbb{R}}

\newcommand{\calA}{\mathcal{A}}
\newcommand{\calB}{\mathcal{B}}

\newcommand{\calH}{\mathcal{H}}

\newcommand{\calK}{\mathcal{K}}

\newcommand{\actson}{\curvearrowright}

\newcommand{\HS}{\calH}
\newcommand{\Hyp}{\hat{\calH}}
\newcommand{\PHyp}{\breve{\calH}}
\newcommand{\dmax}{d_{\max}}

\DeclareMathOperator{\Stab}{Stab}

\DeclareMathOperator{\Aut}{Aut}

\DeclareMathOperator{\diam}{diam}